\documentclass{amsart}
\usepackage{amsfonts,amssymb,amsmath,amsthm}
\usepackage{url}
\usepackage{enumerate}
\usepackage[all]{xy}
\newtheorem{theorem}{Theorem}[section]

\newtheorem{proposition}[theorem]{Proposition}
\newtheorem{corollary}[theorem]{Corollary}
\theoremstyle{definition}
\newtheorem{definition}[theorem]{Definition}

\newtheorem{remark}[theorem]{Remark}
\newtheorem{example}[theorem]{Example}

\numberwithin{equation}{section}

\newcommand{\fp}{\frak{p}}



\begin{document}

\title[Lyubeznik Table of Ideal of Cycle Graphs]{Lyubeznik Table of Ideal of Cycle Graphs}
\author{Parvaneh Nadi}
\email{nadi$_{-}$p@aut.ac.ir}
\author{Farhad Rahmati}
\email{frahmati @ aut.ac.ir}
\address{Department of Mathematics and Computer Science, Amirkabir University of Technology, 424 Hafez Av, Tehran, 1591634311, Iran.}
\author{Majid Eghbali} 
\email{m.eghbali@tafreshu.ac.ir}
\address{Department of
Mathematics,Tafresh University, P.O.Box 39518 79611, Tafresh, Iran} 
  \date{}
\maketitle
\begin{abstract}
Let $R = K[x_1, \ldots, x_n]$ be a polynomial ring over a field $K$,
and $I:=I_{C_n} \subset R$ be an edge ideal of $n$-cycle graph
${C_n}$. In the present paper, we compute the last column of the
Lyubeznik table of $R/I$. 
\end{abstract}
\section{Introduction}
\vskip 0.4 true cm
\label{sect1}
\noindent
Throughout this paper all rings are commutative and Noetherian. Let
$A$ be a local ring which admits a surjective ring homomorphism $\pi
:  S \rightarrow A$, where $(S,\eta,K)$ is a regular local ring of
dimension $n$ containing a field. Set $I = \mathrm{Ker}(\pi)$. In 1993,
Lyubeznik \cite{Lyu2}, using $D$-modules, showed that the local cohomology module
$H_{\eta}^{i}(H_{I}^{n-j}(S))$ is injective and supported at $\eta$.
In these cases $H_{\eta}^{i}(H_{I}^{n-j}(S))$ is a finite direct sum
of some  ($=\lambda_{i,j}(A)$ known as the $i,j$ \textit{Lyubeznik
number} of $A$) copies of the injective hull $E(S/\eta)$ of the
residue field of $S$. In fact the Lyubeznik number
$\lambda_{i,j}(A)$, $i,j \geq 0$ is $i$-th Bass number of the local
cohomology module $ H_{I}^{n-j}(S) $ as:
\[ \lambda_{i,j}(A) := \mu_{i}(\eta,
H_{I}^{n-j}(S))=\mathrm{dim}_{K}\mathrm{Ext}^{i}_{S}(K,H_{I}^{n-j}(S))=
\mu_{0}(\eta,H_{\eta}^{i}(H_{I}^{n-j}(S))).\]
Thus, it depends only on $A$, $i$ and $j$.
 These numbers satisfy the following properties:
\begin{itemize}
\item[i)]
$\lambda_{i,j}(A) =0 $ if $ j > d$,
\item[ii)]
$\lambda_{i,j}(A) =0$ if $i > j$ and $\lambda_{d,d}(A) \neq 0$,
\item[iii)]
Euler characteristic (cf.\cite{Al}): $\sum\limits_{0 \leq i,j \leq
d} (-1)^{i-j} \lambda_{i,j}(A)=1$,
\end{itemize}
where $d = \mathrm{dim} A$. Therefore, we collect all nonzero
Lyubeznik numbers as follows:
\begin{center}
$\Lambda(A)=\begin{bmatrix}
 \lambda_{0,0}&...&\lambda_{0,d} \\
 & \ddots& \vdots\\
& &\lambda_{d,d}
\end{bmatrix},$
\end{center}
in the so-called \textit {Lyubeznik table}. The Lyubeznik table of
$A$ is trivial if $\lambda_{d,d}(A)=1$ and the rest of these
invariants vanish.
Let $R=K[x_1,...,x_n]$ be a polynomial ring over a field $K$ and $G$ a
graph with the vertex set $V(G)=\{x_1,...,x_n \}$ and edge set
$E(G)$. One can associate with $G$ a monomial ideal of R which is
generated by $\{x_ix_j : \{x_i,x_j\} \in E(G) \}$, called the
monomial \textit{edge ideal} of $G$ denoted by $I_{G}$. (See
\cite{V} for more details).\\
\indent To do so further, in the present paper, we are interested in
examining the Lyubeznik table of ideal of cycle graphs especially the last column of these tables. 
For this purpose, in Section
\ref{sec:2}, by using the concept of the minimal vertex covers for $C_n$ of
minimum cardinality we will get some basic results.
As it has obtained from \cite{Al}, when $I$ is a  canonical Cohen-Macaulay squarefree monomial ideal we only have the number $\lambda_{d,d}$ on the last column of it's table.
 In Section  \ref{sec:3}, we examine the canonical Cohen-Macaulay property of $R/I_{C_n}$. Finally, in Section \ref{sec:4}, we compute the
Lyubeznik table of the unmixed part of $I_{C_n}$ and the last column of the
Lyubeznik table of $R/I_{C_n}$.
\section{ Basic Results}
\noindent
\label{sec:2} Let  $R=K[x_1,...,x_n]$ be a polynomial ring over a
field $K$,  $\mathfrak{m}$ denotes its homogeneous maximal ideal $(x_1, . . . ,
x_n)$. Let $C_n$, the \textit{$n$-cycle graph}, be a graph on the
vertex set $\{x_1,...,x_n\}$ and the edge set $\{\{x_1,x_2\},
\{x_2,x_3\}, ..., \{x_{n-1},x_n \}, \{x_n,x_1 \}\}$. Then the edge
ideal of $C_n$ is the ideal \begin{center}
$I_{C_n}=(x_1x_2,x_2x_3,...,x_{n-1}x_n,x_nx_1).$
\end{center}

\begin{definition}\label{0}
 A subset $S \subset V(G)$ is a \textit{vertex cover} for $G$ if for all $\{x,y\} \in E(G), x \in S$ or $y \in S$. The subset $S$ is called a \textit{minimal vertex cover} for $G$ if there is no proper subset of $S$ which is a vertex cover.
\end{definition}

\begin{proposition}(cf. \cite[Proposition 6.1.16]{V})\label{1}
A prime ideal $\mathfrak{q}$ of $R$ is a minimal prime of $I_{G}$ if and only
if $\mathfrak{q}$ is generated by a minimal vertex cover for $G$.
\end{proposition}
\begin{remark}\label{minimal vertex}
Let $n$ be an even integer. Then only minimal vertex covers for
$C_n$ of cardinality $n/2$ are
$S_1=\{x_1,x_3,x_5,...,x_{n-3},x_{n-1}\}$ and $S_2=\{x_2,x_4,...,x_{n-2},x_n\}$.
Also, let $s$ be a positive integer and $n=2s+1$,  then the
minimal vertex covers for $C_n$ of cardinality $s+1$ are,
\begin{center}
$S_{i}=\{x_{2r+1}: i-1 \leq r \leq s\}\cup \{x_{2r}: 1 \leq r \leq
i-1 \}$, for all $1 \leq i \leq s+1$  and\\
$R_{j}=\{x_{2r+1}: 0 \leq r \leq j-1\}\cup \{x_{2r}: j \leq r \leq s
\}$, for all $1 \leq j.$
\end{center}
\end{remark}
\begin{proposition}\label{5}
Let $r$ be a positive integer. Let  $ \mathfrak{p}_j=(x_{j},
x_{j+1},...,x_{j+h-1})$
 for $1 \leq j \leq r$, be a monomial prime ideal of $R$ of height $h\geq r$. Let $I=\cap^{r}_{j=1} \mathfrak{p}_j$, then $\mathrm{depth} R/I=n-h$.
\end{proposition}
\begin{proof}[proof]
We proceed by induction on $r$. Let $r=2$.
Then the Depth Lemma \cite[Lemma 1.3.9]{V}, implies $\mathrm{depth} (R/\mathfrak{p}_1\oplus R/\mathfrak{p}_2)=n-h$, since  $\mathrm{depth} R/\mathfrak{p}_1=\mathrm{depth} R/\mathfrak{p}_2= n-h$.\\
\indent Moreover, we  know that $\mathfrak{p}_1+\mathfrak{p}_2=(x_{1},...,x_{h},x_{h+1})$
and $\mathrm{depth} R/(\mathfrak{p}_1+\mathfrak{p}_2)=n-h-1<n-h$. By applying the Depth
Lemma to the following exact sequence
\[0 \rightarrow R/(\mathfrak{p}_1\cap \mathfrak{p}_2) \rightarrow R/\mathfrak{p}_1\oplus R/\mathfrak{p}_2
\rightarrow R/(\mathfrak{p}_1+\mathfrak{p}_2) \rightarrow 0,\]
 we have,
$\mathrm{depth} R/(\mathfrak{p}_1\cap \mathfrak{p}_2)=\mathrm{depth} R/(\mathfrak{p}_1+\mathfrak{p}_2)+1=n-h$.
 Now, suppose that the statement is true for $r-1$. We show that it is true for $r$. Let $J=\cap^{r-1}_{j=1}\mathfrak{p}_{j}$, by using induction hypothesis, $\mathrm{depth} R/J= n-h$.
It should be noted that  $\mathfrak{p}_{j+1}+\mathfrak{p}_{r}\subset \mathfrak{p}_{j}+\mathfrak{p}_{r}$, for
every $1\leq j \leq r-2$ because
$\{x_{j+1},...,x_{j+h-1} \}\subset\{x_{j},...,x_{j+h-1}\}$,
 and $x_{j+h} \in \mathfrak{p}_{r}$ since $r+j\leq h+j<r+h-1$. On the other hand $J+\mathfrak{p}_{r}=\cap^{r-1}_{j=1}(\mathfrak{p}_{j}+\mathfrak{p}_{r})$, by \cite[Lemma 2.7]{I-SW}.
Therefore
 $J+\mathfrak{p}_{r}=\mathfrak{p}_{r-1}+\mathfrak{p}_{r}$. Since $\mathfrak{p}_r, \mathfrak{p}_{r-1}$ differ by one generator, we get
$\mathrm{depth} R/(J+\mathfrak{p}_{r})=n-h-1$.
Once again using the Depth Lemma to exact sequence,
\[0 \rightarrow R/(\mathfrak{p}_r\cap J) \rightarrow R/\mathfrak{p}_r\oplus R/J \rightarrow
R/(\mathfrak{p}_r+J) \rightarrow 0,\]
we prove the claim.
\end{proof}

\begin{proposition}\label{6}
Keep the notations of Remark \ref{minimal vertex}. Let
$I_1=\cap^{s+1}_{i=1}\mathfrak{p}_{i}$, $I_2=\cap^{s}_{j=1}\mathfrak{q}_{j}$ where
$\mathfrak{p}_{i}$, $\mathfrak{q}_{j}$ are the prime ideals of $R$ generated by $S_{i}$
and $R_{j}$ respectively. Then
 \[I_1+I_2=(\mathfrak{p}_1+\mathfrak{q}_s)\cap (\mathfrak{q}_1+\mathfrak{p}_{s+1}).\]
\end{proposition}

\begin{proof}[proof]
By  \cite[Lemma 2.7]{I-SW}, we may have $I_1+I_2=\cap^{s+1}_{i=1}
\mathfrak{p}_i+\cap^{s}_{j=1}\mathfrak{q}_j=\cap^{s+1}_{i=1}\cap^{s}_{j=1}(\mathfrak{p}_i+\mathfrak{q}_j)$. So
it is enough to show that for all $i,j,~~ \mathfrak{p}_1+\mathfrak{q}_{s}\subseteq
\mathfrak{p}_i+\mathfrak{q}_j$ or $\mathfrak{q}_1+\mathfrak{p}_{s+1}\subseteq \mathfrak{p}_i+\mathfrak{q}_j$.
For the convenience of the reader we sort them below:
\begin{flushleft}
$\mathfrak{p}_1=(\{x_{2r+1}: 0\leq r \leq s \})$, $\mathfrak{p}_i =(\{x_{2r+1}: i-1 \leq r \leq s \} \cup \{x_{2r}: 1 \leq
r \leq i-1 \})$
\\ $\mathfrak{q}_1=(\{x_{2r}: 1 \leq r \leq s \}\cup \{x_1\})$, $\mathfrak{q}_j= (\{x_{2r+1}: 0 \leq r \leq j-1 \} \cup \{x_{2r}: j \leq r \leq s \})$
$\mathfrak{p}_{s+1}=(\{x_{2s+1} \} \cup \{x_{2r}: 1 \leq r \leq s \})$, $\mathfrak{q}_s=(\{x_{2r+1}: 0 \leq r \leq s-1 \}\cup \{x_{2s} \}).$
\end{flushleft}
\indent Let $i < j$. Then we show that $\mathfrak{p}_1+\mathfrak{q}_{s}\subseteq \mathfrak{p}_i+\mathfrak{q}_j$.
It is clear to see that,  $S_1- \{x_{2r+1}: 0 \leq r
\leq i-1 \} \subseteq S_i$  and $\{x_{2r+1}: 0 \leq r \leq i-1 \}
\subseteq R_j$.
So $\mathfrak{p}_1 \subseteq \mathfrak{p}_i + \mathfrak{q}_j$.
Also $R_s - \{x_{2r}\} \subseteq S_1$ and $x_{2r} \in R_j$ therefore, $\mathfrak{p}_1+\mathfrak{q}_{s}\subseteq \mathfrak{p}_i+\mathfrak{q}_j$.\\
\indent For $i \geq j$, $\mathfrak{q}_1+\mathfrak{p}_{s+1}\subseteq \mathfrak{p}_i+\mathfrak{q}_j$, as
 $S_{s+1}-\{x_{2r}: i-1 \leq r \leq s \} \subseteq
S_i$ and $\{x_{2r}: i-1 \leq r \leq s \} \subseteq R_j$.
 Hence $\mathfrak{p}_{s+1} \subseteq \mathfrak{p}_i + \mathfrak{q}_j$. Also $\mathfrak{q}_1 \subseteq \mathfrak{p}_i + \mathfrak{q}_j$ since $R_1 -\{ x_1\} \subseteq S_{s+1}$ and $x_1 \in R_j$.
\end{proof}

\begin{example}\label{e1}
Let $n=7$ and $ I_{C_7}=I_1\cap I_2$ such that,
\begin{eqnarray*}
I_{1}&=(x_1,x_3,x_5,x_7)\cap (x_2,x_3,x_5,x_7)
\cap (x_2,x_4,x_5,x_7) \cap  (x_2,x_4,x_6,x_7) 
\\
I_{2}&= (x_1,x_2,x_4,x_6) \cap (x_1,x_3,x_4,x_6) \cap (x_1,x_3,x_5,x_6).
\end{eqnarray*}
Therefore,  $\mathrm{depth} R/I_1=\mathrm{depth} R/I_2=3$. and $I_1+I_2=(x_1,x_3,x_5,x_6,x_7)\cap(x_1,x_2,x_4,x_6,x_7)$.
\end{example}

\section{ CCM Property of Edge Ideal of Cycle Graphs}
\label{sec:3}
\noindent
Let $(A,\eta)$ be an $n$-dimensional Gorenstein local ring and $M$
be an $A$-module.
 The module $K_{M}=\mathrm{Ext}^{n-\mathrm{dim}(M)}_{A}(M,A)$ is called the canonical module of $M$.
We say that $M$ is canonically Cohen-Macaulay (abbreviated by CCM) if $K_{M}$ is Cohen-Macaulay (CM for short).\\
\indent In this section we are going to examine the canonical Cohen-Macaulay property of a cycle graph.
For our purpos, we need to recall the concept of the unmixed part of an ideal and Buchsbaum modules.\\
\indent Let $I$ be an ideal of $A$. $I^u$ denotes the
 the unmixed part of I, that is, the intersection of those primary
components $Q$ of $I$ with $\mathrm{ht}(Q) = \mathrm{ht}(I)$. The ideal $I$ is unmixed if $I=I^u$. \\
\indent In the course of this section we will also assume that $R=K[x_1,...,x_n]$ be a polynomial ring over a
field $K$,  $\mathfrak{m}$ denotes its homogeneous maximal ideal. Let $I$ be the monomial edge ideal of cycle graph $C_n$.

\begin{remark}\label{un}
By considering the notation of Proposition \ref{6} for $n=2s+1$, 
$I^{u}= I_1\cap I_2$ such that $I_1$ and $I_2$ are both CM of dimension $s$ by Proposition \ref{5}. Also for $n=2s$ we have $I^{u}=\mathfrak{p}_1\cap \mathfrak{p}_2$ where $\mathfrak{p}_1$ and $\mathfrak{p}_2$ respectively are prime ideals generated by $S_1$ and $S_2$ in Remark \ref{minimal vertex}. One can see that by the depth lemma,
 \begin{align*}
\mathrm{depth} R/I^{u}= \left\{
\begin{array}{rl}
1 & \text{if } n=2s\\
2 & \text{if}  ~n=2s+1
\end{array}. \right.
\end{align*}
\end{remark}

\begin{definition}
Let $(A,\eta)$ be a local ring and $M$ be a Noetherian $R$-module.
We say that a sequence of elements $x_1,...,x_t$ in $\eta$ is a $\mathfrak{p}$-week $M$-sequence, if for $i=1,..,t$,
\begin{center}
$(x_1,...,x_{i-1}). M: x_i \subseteq (x_1,...,x_{i-1}). M: \mathfrak{p}$.
\end{center}
where $\mathfrak{p}$ is an ideal of $A$. $M$ is a Buchsbaum module if every system of parameters
of $M$ is a weak $M$-sequence.
\end{definition}
\begin{proposition}\label{B} 
Let $(A,\eta)$ be a CM ring.
Suppose that $\mathfrak{p}_{1},\mathfrak{p}_{2}$ are both CM ideals of $A$ of
dimension $d>1$
and $\mathfrak{p}_{1}+\mathfrak{p}_{2}=\eta$. Then
$A/(\mathfrak{p}_{1}\cap \mathfrak{p}_{2})$ is a Buchsbaum module. Furthermore,
 $A/(\mathfrak{p}_{1}\cap \mathfrak{p}_{2})$ is a CCM module.
\end{proposition}
\begin{proof}[proof]
Let $\fp=\mathfrak{p}_{1}\cap \mathfrak{p}_{2}$. From the long exact sequence
\begin{center}
$0\rightarrow H_{\eta}^{0}(A/\eta)\rightarrow
H_{\eta}^{1}(A/\fp)\rightarrow
0\rightarrow...\rightarrow 0\rightarrow H_{\eta}^{d}(A/\fp)\rightarrow H_{\eta}^{d}(A/\mathfrak{p}_{1}\oplus A/\mathfrak{p}_{2})\rightarrow 0$
\end{center}
one has
$H_{\eta}^{i}(A/\fp)=0$ for all $1<i<d$ and  $H^{1}_{\eta}(A/\fp)\cong
H_{\eta}^{0}(A/\eta)=A/\eta$,
because $A/\mathfrak{p}_{1}$ and $A/\mathfrak{p}_{2}$ are both CM of dimension
$d$. So $\eta H^{1}_{\eta}(A/\fp)=0$. Now, we are done by \cite[Proposition 2.12, page 81]{St-V}.
For the last statment, note that for all $1<i<d$,
$H_{\eta}^{i}(A/\fp)=0$. Hence, the claim follows  from
\cite[Theorem 5.4]{G-Sh}.
\end{proof}
\begin{proposition} \label{Buchs}
Let $I$ be the monomial edge ideal of the $n$-cycle graph $C_{n}$. Then 
 $R/I^{u}$ is a Buchsbaum module.
\end{proposition}

\begin{proof}[proof]
According to what was said in the Remark \ref{un}, for even integer $n$, $R/I^{u}$ is a Buchsbaum module by  Proposition
\ref{B}.
 In the case $n$ is odd, if $n=3~ or~ 5$ then, $R/I$ is CM so it is Buchsbaum by \cite[Example 2]{St-V}.
  Therefore, we consider $n\geq 7$ and set $d=\mathrm{dim} R/I$. We have $I^{u}=I_{1}\cap I_{2}$ such that
   $R/I_{1}$ and $R/I_{2}$ are both CM modules. Also we can extract of  Proposition \ref{6}, $R/(I_{1}+I_{2})$ is a Buchsbaum module by  Proposition \ref{B}. Put $J=I_{1}+I_{2}$. We have the long exact sequence
\begin{center}
$0\rightarrow H_{\mathfrak{m}}^{1}(R/J)\rightarrow
H_{\mathfrak{m}}^{2}(R/I^{u})\rightarrow
0\rightarrow...\rightarrow 0\rightarrow
H_{\mathfrak{m}}^{d-1}(R/J)\rightarrow H_{\mathfrak{m}}^{d}(R/I^{u})\rightarrow H_{\mathfrak{m}}^{d}(R/I_{1}\oplus R/I_{2})\rightarrow
0.$
\end{center}
As for all $ 1<i<d-1$, $H_{\mathfrak{m}}^{i}(R/J)=0$ and also by the Remark \ref{un}, $\mathrm{depth}(R/I^{u})=2$, we can conclude that for
all $2< i <d$, $ H^{i}_{\mathfrak{m}}(R/I^{u})=0$.
On the other hand, $H_{\mathfrak{m}}^{1}(R/J)\cong H_{\mathfrak{m}}^{2}(R/I^{u})$ and by \cite[Proposition 2.12, page 81]{St-V}, $\mathfrak{m} H^{1}_{\mathfrak{m}}(R/J)=0$. Hence $\mathfrak{m} H^{2}_{\mathfrak{m}}(R/I^{u})=0$. Once again applying \cite[Proposition 2.12, page 81]{St-V}, $R/I^{u}$ is Buchsbaum.
\end{proof}
\begin{remark}\label{ccm}
We note that for any ideal $J$ of $R$ we may write $J=J^{u}\cap J'$ and get the following long exact sequence
 \begin{center}
$...\rightarrow \mathrm{Ext}_{R}^{n-d}(R/(J^{u}+J'),R)\rightarrow
\mathrm{Ext}_{R}^{n-d}(R/J^{u},R)\oplus
\mathrm{Ext}_{R}^{n-d}(R/J',R)\rightarrow
\mathrm{Ext}_{R}^{n-d}(R/J,R)\rightarrow
\mathrm{Ext}_{R}^{n-d+1}(R/(J^{u}+J),R)\rightarrow...$,
\end{center}
where $d=\mathrm{dim} R/J$. Since $\mathrm{dim} R/(J^{u}+J')<\mathrm{dim} R/J'<d$,
 \begin{center} $
\mathrm{Ext}_{R}^{n-d}(R/(J^{u}+J'),R)=\mathrm{Ext}_{R}^{n-d}(R/J',R)=\mathrm{Ext}_{R}^{n-d+1}(R/(J^{u}+J'),R)=0.$
\end{center}
It follows that $K_{R/J}\cong K_{R/J^{u}}$. So
$R/J$ is CCM if and only if $R/J^{u}$ is CCM.
\end{remark}

\begin{proposition}
Let $n\neq 3 , 5$ and $I$ be the monomial edge ideal of the $n$-cycle graph $C_{n}$. Then
 $R/I$ is CCM module if and only if $n$ is an even integer.
\end{proposition}
\begin{proof}[proof]
Let $n$ be an even integer by the Proposition \ref{B} and the Remark \ref{ccm}, $R/I$ is CCM.

Conversely, let $R/I$ be CCM. So $R/I^{u}$ is. If $n$ is not even, as $R/I^{u}$ is Buchsbaum,
for $1 <i < d=\mathrm{dim} R/I$, $H^{i}_{\mathfrak{m}} (R/I^{u})=0$ by \cite[Theorem 5.4]{G-Sh}. It contradicts this fact that
$\mathrm{depth} R/I^{u}=2$. So $n$ is even and we are done.
\end{proof}

\section{ Lyubeznik Numbers}
\label{sec:4}
\noindent
We assume as in the previous section that $R=K[x_1,...,x_n]$ is a
polynomial ring over a field $K$,  $\mathfrak{m}$ denotes its homogeneous
maximal ideal $(x_1, . . . , x_n)$. The aim of this section is to
describe the last column of the Lyubeznik table of $R/I$, where $I$
is the monomial edge ideal of the $n$-cycle graph $C_{n}$. For this we will establish the Lyubeznik
table of the unmixed part of cyclic ideals.
\begin{corollary}
Let $n$ be an even integer. Then the Lyubeznik
table of $R/I^{u}$ is as follows:
\begin{center}
$\Lambda(R/I^{u})=\begin{bmatrix}
 0&1&0&\cdots &0 \\
 &0 &\ddots & \vdots&\vdots\\
& &\ddots &&0\\
&& &&2\\
\end{bmatrix}.$
\end{center}
\end{corollary}
\begin{proof}[proof]
By \cite[Theorem 5.5]{N-R-E}, the Lyubeznik table of  $R/I^{u}$ is as 
\begin{equation}\label{table}
\Lambda(S/J)=\begin{bmatrix}
0&\lambda_{0,1}&\lambda_{0,2}&\cdots &\lambda_{0,d-1} & 0 \\
 &0 & 0&\ldots & 0 &0\\
  & & 0&\ldots & 0 &\lambda_{0,d-1}\\
  & & &\ddots & \vdots&\vdots\\
  & & & & 0 & \lambda_{0,2}\\
  & & & &  & \lambda_{0,1}+1\\
\end{bmatrix}.
\end{equation}
 On the other hand by \cite{Y}, for any squarefree monomial ideal $J$ of $R$,
\begin{center}
$\lambda_{i,j}(R/J)=\mathrm{dim}_{k}[\mathrm{Ext}^{n-i}_{R}(\mathrm{Ext}^{n-j}_{R}(R/J,R),R)]_{0}.$
\end{center}
So from the Proposition \ref{B} and local duality, for all $1<i<d$,  $\lambda_{0,i}(R/I^{u})=0$ since $H_{\mathfrak{m}}^{i}(R/I^{u})=0$. Also
\begin{center}
$\lambda_{0,1}(R/I^{u})=\mathrm{dim}_{K}[\mathrm{Ext}^{n}_{R}(\mathrm{Ext}^{n-1}_{R}(R/I^{u},R),R)]_{0}= \mathrm{dim}_{K}[\mathrm{Ext}^{n}_{R}(K,R)]_{0}=1$,
\end{center}
because  $H^{1}_{\mathfrak{m}}(R/I^{u})\cong
H_{\mathfrak{m}}^{0}(A/\mathfrak{m})=A/\mathfrak{m}$. So $\lambda_{d,d}(R/I^{u})=2$ where $d=\mathrm{dim}(R/I)$ and we get the result.
\end{proof}
\begin{remark}
 Note that for $n=3,5$ the ideal $I$ is Cohen-Macaullay so it has the
trivial Lyubeznik table. Hence we compute the
Lyubeznik table of $R/I^{u}$ for $n=2s+1$ where $s \geq 3$.
\end{remark}
\begin{corollary}
Let $n=2s+1$ where $s \geq 3$. Then the Lyubeznik table of $R/I^{u}$ is as follow:

\begin{center}
$\Lambda(R/I^{u})=\begin{bmatrix}
0&0&1&\cdots &0 \\
 &0 &\ddots & \vdots &\vdots\\
& &\ddots &&1\\
&& &&1\\
\end{bmatrix}.
$
\end{center}
\end{corollary}
\begin{proof}[proof]
The Lyubeznik table of $R/I^{u}$ is as form (\ref{table}). By the proof of proposition \ref{Buchs}, for $2<i <d$,  $\lambda_{0,i}(R/I^{u})=0$ as $H^{i}_{\mathfrak{m}}(R/I^{u})=0$.
Also $H^{2}_{\mathfrak{m}}(R/I^{u})\cong H^{1}_{\mathfrak{m}}(R/I_1+I_2)$. So by local duality
$[\mathrm{Ext}^{n}_{R}(\mathrm{Ext}^{n-2}_{R}(R/I^{u},R),R)]_{0}\cong [\mathrm{Ext}^{n}_{R}(\mathrm{Ext}^{n-1}_{R}(R/I_1+I_2,R),R)]_{0}$. Hence $\lambda_{0,2}(R/I^{u})=\lambda_{0,1}(R/I_1+I_2)=1$. Therefore $\lambda_{d-1,d}(R/I^{u})=1$ for $d=\mathrm{dim}(R/I)$ and the proof is complete.
\end{proof}
Now by using this fact that $K_{R/I}\cong K_{R/I^{u}}$, as we saw in the Remark \ref{ccm}, we can get our main results.
\begin{corollary}
Let $n$ be an even integer. Then the last column of the Lyubeznik
table of $R/I$ is as follows:
\begin{center}
$[0, 0, \dots, 0, 2].$
\end{center}\end{corollary}
\begin{corollary}
Let $s \geq 3$ be a positive integer and $n=2s+1$. Then the last column of
Lyubeznik table of $R/I$ is
\begin{center}
$[0,0, \dots,0, 1,1].$
\end{center}
\end{corollary}
\begin{example}
Let $I=I_{C_7}$ and $J=I_{C_4}$ then
\begin{center}
$\Lambda(R/I)=\begin{bmatrix}
 0&0&1&0 \\
 &0&0 & 0\\
&&0 &1\\
&& &1\\
\end{bmatrix},$~~~~
$\Lambda(R/J)=\begin{bmatrix}
0&1&0 \\
&0&0\\
&&2\\
\end{bmatrix}.$
\end{center}
\end{example}


\proof[\bf Acknowledgment]
The first author would like to thank Professor Dariush Kiani for some helpful discussions.






\end{document}